\documentclass[12pt,reqno]{amsart}

\usepackage{lib} 
\usepackage[colorlinks=true,hyperindex,pagebackref=false, citecolor=cyan,linkcolor=cyan]{hyperref}
\usepackage{tikz-cd}
\usepackage{enumitem}
\usepackage{mathrsfs} 
\usepackage{mathtools}
\mathtoolsset{showonlyrefs} 
\usepackage[margin=1.4in]{geometry}

\makeatletter
\newtheorem*{rep@theorem}{\rep@title}
\newcommand{\newreptheorem}[2]{%
\newenvironment{rep#1}[1]{%
 \def\rep@title{#2 \ref*{##1}}%
 \begin{rep@theorem}}%
 {\end{rep@theorem}}}
\makeatother
\newreptheorem{theorem}{Theorem}

\title{The moduli space of $G$-algebras}

\author{Andrew O'Desky}
\address{
\parbox{0.7\linewidth}{
    Department of Mathematics\\
    Princeton University\\
    Princeton, NJ 08544-1000, USA\\[.1em]
}}
\email{andy.odesky@gmail.com}

\author{Julian Rosen}
\address{
\parbox{0.7\linewidth}{
    Department of Mathematics and Statistics\\
    University of Maine\\
    Orono, ME 04469\\
    }
}
\email{julianrosen@gmail.com}

\date{\today} 

\begin{document}

\begin{abstract} 
    Let $L$ be a Galois algebra with Galois group $\gp$ 
    and let $x$ be a normal element of $L$. 
    The moduli space $\X$ of pairs $(L,x)$ 
    is isomorphic to an open subset of 
    the quotient variety $\PPP/\gp$, 
    where $\PPP$ is the projective space 
    of the regular representation of $\gp$. 
    We provide a formula for the height of any pair 
    $(L,x) \in \Xg(\QQ)$ 
    in terms of algebraic invariants of $L$ and $x$ 
    with respect to a natural adelic metric on 
    the anticanonical divisor of $\PPP/\gp$.  
\end{abstract} 

\subjclass[2020]{11G50, 14D22, 14M17, 11G35}
\keywords{heights, moduli spaces, normal elements.}

\maketitle

\section{Introduction} 

Let $G$ be a finite group acting linearly on 
a finite-dimensional rational vector space $V$, 
and let $\PP = \PP(V)$ denote the projective space of $V$. 
In this paper we study 
heights of rational points on $\PP/G$. 

Our main result is that for $V$ the regular representation of $G$ 
and a certain open subset $\X$ of $\PP/G$, 
the height of any rational point $P$ on $\X$ 
can be expressed in terms of certain objects attached to $P$. 
We show that $P$ determines, and is determined by, 
a Galois $\QQ$-algebra $L$ with Galois group $G$ 
and a trace-one normal element $x$ of $L$ 
up to Galois conjugacy. 
The Galois algebra $L$ is isomorphic to 
    $\nfe \times \cdots \times \nfe$ 
    for a number field $\nfe$. 
The images of the Galois conjugates of $x$ in $\nfe$ 
generate a fractional $K$-ideal $J$. 
Let $\norm{\cdot}_\infty$ denote the canonical norm 
on Minkowski space $L_\RR$. 

\begin{theorem}\label{thm:thm1}
Let $\phi \colon \PPP/G \to \PP^N$ be a non-constant morphism. 
Let $d_\phi$ be the degree of the composite map 
    $\PPP \to \PPP/\gp \xrightarrow{\phi} \PP^N$. 
Then for all $\pt = (L,x) \in \X(\QQ)$, 
\begin{equation}\label{eqn:newmainformula}
    h(\phi(\pt)) =
    d_\phi\log{\norm{x}_\infty} - 
    \frac{d_\phi}{[\nfe:\QQ]}\log{N(J)} + O(1) 
\end{equation}
    for a bounded function $O(1)$. 
\end{theorem}

To prove this theorem we use the method of descent 
to construct an ample line bundle $\DLB$ on $\PPP/\gp$ 
which is linearly equivalent to 
the anticanonical divisor of $\PPP/\gp$ 
up to torsion in the divisor class group. 
We show that $\DLB$ is globally generated 
    and restricts to an immersion on $\X$. 
These geometric results are used to prove 
    Theorem~\ref{thm:thm1} in \S\ref{sec:heightformulas}. 

In the last section we give a sharper result. 
We prove that the function 
\begin{equation}\label{eqn:acheightintro}
    h_{AC}(L,x) = 
    |G|\log{\norm{x}_\infty} - 
    \frac{|G|}{[\nfe:\QQ]}\log{N(J)}
\end{equation}
(with no $O(1)$) 
is actually the height function $h_{AC}$ 
associated with a natural adelic metric on $\DLB$. 

\subsection{Self-dual elements} 

In addition to normal elements, 
    we also consider \emph{self-dual} elements. 
Recall an element $x \in L$ is self-dual if for all $g \in G$ 
    we have 
\begin{equation}
    \mathrm{tr}^L_\QQ(x g(x)) = 
    \begin{cases}
        1 & \text{if $g = 1$,}\\
        0 & \text{otherwise.}
    \end{cases} 
\end{equation}
Here we assume that $G$ has odd order 
since this guarantees the existence of self-dual elements 
by a theorem of Bayer-Fluckiger--Lenstra~\cite{bayer_lenstra}. 

The formula \eqref{eqn:newmainformula} 
indicates that rational points of $\X$ 
corresponding to self-dual elements 
are of particular interest. 
Suppose $L = \nfe$ for simplicity, and 
let $x \in \nfe$ be any self-dual element. 
The first indication is that 
$$\norm{x}_\infty = \sqrt{\tr^\nfe_\QQ(x^2)} = 1.$$ 
Indeed the left equality holds since $\nfe$ is totally real, 
and the right since $x$ is self-dual. 
The second indication involves the norm term $N(J)$. 
Recall the basic formula relating norms and discriminants: 
\begin{equation}
    N(J) = [O_\nfe:J] = \sqrt{\disc{J}\disc{\nfe}^{-1}} . 
\end{equation}
In particular, if $J$ is {unimodular}, 
then $\disc{J} = 1$ and we have the appealing formula 
$$h_{AC}(\nfe,x) = \log \sqrt{\disc{\nfe}}.$$ 
Naturally it would be of interest to have a direct relationship 
    between heights and discriminants 
    (cf.~e.g.~\cite{ESZB}, \cite{yasuda-2015}) 
    so we should like to understand 
    the extent to which $J$ fails to be unimodular. 
As $x$ is self-dual the smaller lattice 
\begin{equation}
    \lat{}\coloneqq
    \sum_{g \in G}
    \ZZ g(x) 
\end{equation}
is unimodular, however in general 
\begin{equation}
    \sum_{g \in G}
    \ZZ g(x) 
    \neq 
    \sum_{g \in G}
    O_\nfe g(x) = J. 
\end{equation}

Thus we ask whether 
there is a relationship between $N(I)$ and $N(J)$. 
In this direction we prove the following result. 
Let $K/\QQ$ be a Galois field extension with Galois group $G$ 
and let $x$ be a normal element of $K$, not necessarily self-dual. 
For each integer $D \geq 1$ 
consider the following order of $\nfe$: 
\begin{equation} 
    T_D = \{a \in \nfe : a I^D \subset I^D\}. 
\end{equation} 
These orders increase with $D$ 
and stabilize, and we call 
\begin{equation}\label{eqn:stablemultiplierorder}
T_\infty \coloneqq \lim_{D \to \infty} T_D 
\end{equation}
the \emph{stable multiplier order} of $I$. 
We prove two stability results 
and also find another interpretation for $T_\infty$. 

\begin{theorem}\label{thm:stablemultiplierorder} 
    \,\,\,
    \begin{enumerate}
        \item $T_\infty = T_{D}$ if $D \geq |G|-1$. 
        \item $N(\lat{}^D) = N(J^D)$ if $D \geq |G|-1$. 
        \item $\Spec T_\infty$ is isomorphic to 
            the fiber of $\PP \to \PP/G$ 
            over the integral point of $\PP/G$ 
            determined by $(\nfe,x)$. 
    \end{enumerate}
\end{theorem} 


%




\section{Construction of orbit parametrizations}  
\label{sec:two} 

In this section we construct orbit parametrizations for  
$\gp$-algebras equipped with normal and self-dual bases. 
If $R$ is a commutative ring and 
$\Spec S$ is a $\gp$-torsor over $\Spec R$ 
in the \'etale topology, 
then we call $S$ a \defn{$\gp$-algebra over $R$}. 

\begin{remark}
The normal basis theorem guarantees 
the existence of normal elements if $S$ is a field,  
however normal elements need not exist for an arbitrary $\gp$-algebra 
(e.g.~$S$ might not be free as an $R$-module 
    or there may be local obstructions due to wild ramification). 
Self-dual elements exist for Galois field extensions of odd degree 
in any characteristic \cite{bayer_lenstra}, \cite{bayer-2001} 
but even degree field extensions may not have self-dual elements 
    (e.g. quadratic field extensions with characteristic $\neq 2$). 
\end{remark}

For fixed $R$, we consider pairs $(S,x)$, where $S/R$ is a $\gp$-algebra and $x\in S$ is a normal element. An isomorphism between pairs $(S,x)$ and $(S',x')$ is a $\gp$-equivariant $R$-algebra isomorphism $\varphi\colon S\is S'$ satisfying $\varphi(x)=x'$. Given a $\gp$-algebra $S$ and a ring homomorphism $f\colon R\to R'$, the base extension $S\otimes R'$ is a $\gp$-algebra over $R'$. If $x\in S$ is normal (resp.\ self-dual), then $x\otimes 1\in S\otimes R'$ is normal (resp.\ self-dual).

\begin{definition}
$\Mg$ is the functor taking a commutative ring $R$ to the set of isomorphism classes of pairs $(S,x)$, where $S/R$ is a $\gp$-algebra and $x\in S$ is a normal element. We define $\MMg$ similarly but with $x$ self-dual. 
\end{definition}

These functors are representable by affine schemes over $\Z$, which we construct as subquotients of the group of units in the group algebra of $G$. 
The functor of commutative rings
\[
    R\mapsto\left\{u =\sum_{g \in \gp} a_g [g]\in R[G]^\times:
    \sum_{g\in\gp} a_g = 1\right\}
\]
is representable by an affine group scheme of finite type over $\Z$, 
which we denote by $\ug$. 
There is an anti-involution $u\mapsto \bar{u}$ of $\ug $, given by
$\overline{\sum_{g \in \gp} a_g[g]} 
= \sum_{g \in \gp} a_g[g^{-1}]$, 
and we also consider 
the subgroup scheme $\NU{\gp}\subset \ug$ of norm-one units given by
\[
\NU{\gp}(R)=\left\{u\in \ug(R) :u\bar{u}=1\right\}.
\]
The coordinate ring $\A$ of $\ug$ is the quotient of $\Z[X_g:g\in \gp][\GDET{\gp}^{-1}]$ 
by the principal ideal generated by 
$(\sum_{g\in\gp} X_g - 1)$ 
where $\GDET{\gp}$ is the determinant of the matrix 
with rows and columns indexed by $G$ whose $g,h$ component is $X_{gh}$. 
The coordinate ring $\asd$ of $\NU{\gp}$ is 
the quotient of $\A$ by the ideal 
$({\sum_{h\in \gp } X_{gh}X_{h} - \delta_{g,1}:g\in \gp})$. 


\begin{lemma}
\label{lembij}
Suppose $S/R$ is a $\gp$-algebra. Then there is a bijection from the set of $\gp $-equivariant ring homomorphisms $\varphi{\colon}\A \to S$ to the set of normal elements of $S/R$, taking $\varphi$ to $\varphi(X_1)$. 
Moreover, $\varphi$ factors through $\asd$ 
    if and only if $\varphi(X_1)$ is self-dual. 
\end{lemma}

\begin{proof}
There is a bijection from the set of $\gp $-equivariant ring homomorphisms
    $\varphi{\colon}\Z[\{X_g\}_{g \in G} ]\to S$ 
    to the set of elements of $S$, given by $\varphi\mapsto \varphi(X_1)$. Fix such a $\varphi$, and set $x:=\varphi(X_1)\in S$. Then $\varphi(X_g)=g^{-1}(x)$, and $\varphi(\GDET{\gp} )^2$ is the discriminant of the set $\{\varphi(X_g)\}_g=\{g(x)\}_g$. This discriminant is a unit if and only if $\{g(x)\}_g$ is a basis for $S/R$, and $\varphi$ kills $\sum_g X_g -1$ if and only if $\tr_{S/R}(x)=1$. This proves the first assertion of the lemma.

To prove the second, observe that $\varphi$ factors through $\asd$ if and only if
\[
\sum_{h\in \gp } hg(x)h(x) = 
    \begin{cases}
        1 & \text{if $g = 1$,}\\
        0 & \text{otherwise} 
    \end{cases} 
\]
for all $g\in \gp$, which is precisely the condition that $x$ is self-dual.
\end{proof}

The group $\gp$ is naturally identified with 
a constant subgroup scheme of $\sdg$, 
and thereby acts freely on $\sdg$ and $\ug$. 

\begin{definition}
    $\X$ is the quotient scheme $\UG {\gp}/\gp$, 
    $\Y$ is the quotient scheme $\NU{\gp}/\gp$.
\end{definition}

The following theorem says that the scheme $\X$ is a fine moduli space for $G$-algebras equipped with a normal element, and $\Y$ is a fine moduli space for $G$-algebras equipped with a self-dual element.

\begin{proposition}
\label{threp}
For any commutative ring $R$ there are bijections 
    $\X(R)\is\Mg(R)$ and $\Y(R)\is\MMg(R)$ 
    which are functorial in $R$.
\end{proposition}

\begin{proof}[First proof of Proposition~\ref{threp}] 
The scheme $\X = \ug/G$ represents 
    the stack quotient $[\ug/G]$ 
    as the $G$-action is free. 
An $R$-point of $[\ug/G]$ is, by definition, 
    a $G$-torsor $\Spec S \to \Spec R$ 
    together with a $G$-equivariant morphism 
    $\Spec S \to \ug$. 
The $G$-equivariant morphisms 
    $\Spec S \to \ug$ 
    are in bijection 
    with $G$-equivariant ring homomorphisms 
    $\mcO(\ug) = A \to S$, 
    which are in bijection 
    with normal elements of $S/R$ 
    by Lemma~\ref{lembij}. 
\end{proof} 

Here is a more elementary proof. 

\begin{proof}[Second proof of Proposition~\ref{threp}] 
Since $\ug\to\X$ is a $G$-torsor, $\A/\A^G$ is a $G$-algebra. 
Then Lemma~\ref{lembij} applied to 
    the identity map $\A\to\A$ implies $X_1\in \A$ is 
    a normal element.

Now, fix a ring $R$. There is a function
\begin{align*}
\gamma_1{\colon}\Hom(\A^G,R)&\to \Mg(R)\\
f&\mapsto \lp \A\otimes_{\A^G,f} R,X_1\otimes 1\rp.
\end{align*}
We also define a function $\gamma_2{\colon}\Mg(R)\to \Hom(\A^G,R)$ as follows. Given $(S,x)\in \Mg(R)$, Lemma~\ref{lembij} implies there is a unique $G$-equivariant homomorphism $f{\colon}\A\to S$ such that $f(X_1)=x$. Then $f$ takes $\A^G$ into $S^G=R$, and we define $\gamma_2(S,x) = f\big|_{\A^G}$. We claim that $\gamma_1$ and $\gamma_2$ are inverses.

In one direction, given $f{\colon}\A^G\to R$, we see directly that the natural map $\A\to \A\otimes_{\A^G,f} R$ is $G$-equivariant and takes $X_1$ to $X_1\otimes 1$. It follows that $\gamma_2(\gamma_1(f))$ is the natural map $\A^G\to \A^G\otimes_{\A^G,f}R=R$, so that $\gamma_2(\gamma_1(f))=f$. In the other direction, given $(S,x)\in\Mg(R)$, let $f{\colon}\A\to S$ be the $G$-equivariant homomorphism satisfying $f(X_1)=S$. Then
\[
\gamma_1(\gamma_2(S,x)) = \lp \A\otimes_{\A^G,f|\A^G} R,X_1\otimes 1\rp.
\]
Now, $S$ is an $R$-algebra, so $f$ extends uniquely to an $R$-algebra homomorphism
\[
\tilde{f}{\colon}\A\otimes_{\A^G,f|\A^G} R\to S.
\]
One see directly that $\tilde{f}$ is $G$-equivariant, and that $\tilde{f}(X_1\otimes 1)=f(X_1)=x$. Finally, every $G$-equivariant morphism of torsors is an isomorphism, so we conclude $\gamma_1(\gamma_2(S,x))\simeq (S,x)$. This proves that $\gamma_1$ and $\gamma_2$ are inverse bijections. Moreover, it is clear that $\gamma_1$ and $\gamma_2$ are natural in $R$.

To prove the statement about $\Y(R)$, we observe that $\gamma_1$ takes $\Hom(\asd^G,R)$ into $\MMg(R)$ because $X_1\in \asd$ is self-dual, and $\gamma_2$ takes $\MMg(R)$ into $\Hom(\asd^G,R)$ by 
    Lemma~\ref{lembij}.
\end{proof} 

\begin{remark}
Gundlach~\cite{gundlach} independently constructed $\X$ 
    and its orbit parametrization 
    for $G$-algebras with a normal element. 
Gundlach's construction uses the constraints 
    on the structure constants of a $G$-algebra 
    to cut out $\X$ inside $\AA^{|\gp|^2}$. 
In special cases the varieties $\X$ 
    have appeared before in the literature 
    \cite[\S VI.2]{serre_alg_gps}, \cite{suwa1}, 
    \cite{levi}, \cite{delone-faddeev-1964}, 
    \cite{wright-yukie}, 
    \cite{bhargava-hcliii}, \cite{bhargava-hcliv} 
    and \cite{bhargava-shnidman}. 
See also \cite{poonen-2008} for a related construction. 
\end{remark}

It is well-known that a $G$-torsor is trivial if and only if 
    it admits a section. 
For $G$-torsors $S/R$ obtained by pulling back $\ug \to \X$ 
    along an $R$-point $(S,x)$ of $\X$, 
    the next proposition gives a formula for such a section 
    after a suitable base change. 

\begin{proposition}\label{prop:fiber}  
    Let $(S,x)\in \X (R)$. 
Let $R'$ be an $R$-algebra and suppose there is an $R$-algebra homomorphism 
    $\varphi{\colon}S\to R'$. We have the following diagram: 
\begin{equation*}
\begin{tikzcd}
    &\Spec S \arrow[r] \arrow[d] & 
        \ug \arrow[d] \\
    \Spec R' \arrow[ur,dashed] \arrow[r]&\Spec R \arrow[r,"(S{,}x)"] & \X .
\end{tikzcd}
\end{equation*}
Then $(S\otimes_R R',x\otimes 1)\in \X (R')$ is the image of the $R'$-valued unit 
\[
u=\sum_{g\in \gp } \varphi(g(x))[g^{-1}]\in \ug (R')
\]
under the natural morphism $\ug \to \X $. 
\end{proposition} 

For the proof, we will use the natural action of $\ug$ on 
the homogeneous space $\ug/G = \X$. 
For $u=\sum_g a_g[g]\in \ug (R)$ and $(S,x)\in \X (R)$ 
this action is given by 
\begin{equation}\label{eqn:actionformula} 
u(S,x) = \lp S,\sum_{g\in\gp} a_g g(x)\rp.
\end{equation} 

\begin{proof} 
The set of morphisms $\Spec R' \to \Spec S$ over $\Spec R$ 
    is in bijection with 
    the set of sections of $\Spec S \times_R \Spec R' \to \Spec R'$ 
    by the universal property of the fiber product. 
Thus the existence of $\varphi$ implies that 
    the pullback of the $G$-torsor $S/R$ to $R'$ 
    is isomorphic to the trivial $G$-torsor over $R'$. 
Recall the coordinate ring of the trivial $\gp$-torsor over $R'$ 
    is the $R'$-algebra $\splitalg{R'}$ of set-theoretic functions 
    $f \colon \gp \to R'$ under pointwise operations 
    with $G$-action given by $g(f)(h) = f(hg)$. 
We have the isomorphism of $G$-algebras over $R'$ given by 
\begin{align*}
    S\otimes_R R'&\xrightarrow{\sim} \splitalg{R'}\\
x\otimes r&\mapsto\big[g\mapsto r\varphi(g(x))\big].
\end{align*}
This isomorphism maps $x \otimes 1$ to 
    the function $g \mapsto \varphi(g(x))$, 
    which is equal to $u\chi_{\{1\}}$ where 
    $u$ is the $R'$-valued unit given by 
    $u=\sum_{g\in\gp} \varphi(g(x))[g^{-1}] \in \ug (R')$ 
    and $\chi_{\{1\}}$ is the characteristic function of the singleton $\{1\}$ 
    containing the identity element $1 \in G$. 
In terms of points of $\X$, this means that 
\begin{equation} 
    (S \otimes_R R',x \otimes 1) = 
    \lp \splitalg{R'},g\mapsto \varphi(g(x))\rp = 
    u\lp \splitalg{R'},\chi_{\{1\}}\rp. 
\end{equation} 
The result now follows from the fact that $(\splitalg{R'},\chi_{\{1\}})\in\X(R')$ 
    is the image of $1\in \ug(R')$ under the $\ug$-equivariant quotient morphism $\ug\to \X$. 
\end{proof} 



\section{Descent for \texorpdfstring{$G$}{G}-line bundles}\label{sec:three}  

In this section we prove there is a unique 
line bundle $\DLB$ over $\PPP/\gp$ 
whose pullback to $\PPP$ is equal to $-\CB{\PPP}$ 
(Theorem~\ref{thm:discriminantlinebundle}). 
We show that $\Pic(\PPP/\gp) \cong \ZZ$ and 
$\DLB$ generates the subgroup of index $n/e$ 
where $n$ (resp.~$e$) denotes the order (resp.~exponent) of $\gp$ 
(Theorem~\ref{thm:dlbintrinsicdefn}). 
We also prove that $\DLB$ is globally generated 
and its global sections 
define an immersion of $\X$ into projective space 
(Theorem~\ref{thm:immersion}). 

\subsection{Descent for \texorpdfstring{$G$}{G}-line bundles} 


Let $\algvar$ denote a projective variety 
equipped with an action by a finite group $\gp$, 
all defined over a characteristic zero field $\nf$. 
We assume the $\gp$-action on the structure sheaf of $Y$ 
is $\mathcal O(Y)$-linear. 
Suppose $\algvar$ admits 
the action of another finite group $\Gamma$ 
commuting with the action of $\gp$. 
Consider a $\gp$-line bundle $\vb$ over $\algvar/\Gamma$. 
When is $\vb$ isomorphic to the pullback of 
a line bundle from $\algvar/(\gp \times \Gamma)$? 
Equivalently, when does $\vb$ vanish under the map 
\begin{equation} 
\Pic_G(\algvar/\Gamma) \to 
    \frac{\Pic_G(\algvar/\Gamma)}{\im(\Pic(\algvar/(G \times \Gamma)) \to \Pic_G(\algvar/\Gamma))}? 
\end{equation} 
If $\Gamma$ acts freely on $\algvar/G$, 
the next lemma shows this can be determined 
by pulling back to $\algvar \to \algvar/G$ 
and resolving the question there. 

\begin{lemma}\label{lemma:picardimages}
Suppose $\algvar/G \to \algvar/(G \times \Gamma)$ 
    is a Galois covering with Galois group $\Gamma$. 
Then a $G$-line bundle on $\algvar/\Gamma$ 
    descends to $\algvar/(G \times \Gamma)$ 
    if and only if its pullback to $\algvar$ 
    descends to $\algvar/G$. 
\end{lemma}

Equivalently, the following natural map is injective: 
\begin{equation}\label{eqn:maponpicardimages} 
    \frac{\Pic_G(\algvar/\Gamma)}{\im(\Pic(\algvar/(G \times \Gamma)) \to \Pic_G(\algvar/\Gamma))}
    \to 
    \frac{\Pic_G(\algvar)}{\im(\Pic(\algvar/G) \to \Pic_G(\algvar))}
\end{equation} 



\begin{proof} 
Let $\vb$ be a $G$-line bundle over $\algvar/\Gamma$ 
    which represents some class in 
    the kernel of \eqref{eqn:maponpicardimages}. 
We have diagrams: 
\begin{equation} 
    \begin{tikzcd}[column sep = tiny]
    &\algvar\arrow[dl] \arrow[dr]&\\
    \algvar/G\arrow[dr]&&\algvar/\Gamma \arrow[dl]\\
    &\algvar/(G \times \Gamma)& 
\end{tikzcd}\hspace{40pt}%
\begin{tikzcd}
    &\vb_1&\\
    \vb_2\arrow[ur, mapsto]&&\vb_4 \arrow[ul, mapsto]\\
    &\vb_3\arrow[ul, mapsto]\arrow[ur, mapsto]& 
\end{tikzcd}
\end{equation} 
where $\vb_1$ is the pullback of $\vb$, 
$\vb_2$ is any line bundle which pulls back to $\lb_1$, 
$\vb_3$ is the quotient of $\vb_2$ by $\Gamma$, 
and $\vb_4$ is the pullback of $\vb_3$ 
    (the quotient of $\vb_2$ is 
    for its $\Gamma$-linearization coming from $\lb_1$, 
    and the quotient exists by descent along torsors 
    since $\Gamma$ acts freely). 
As the left diagram commutes, $\vb_4$ pulls back to $\vb_1$ in $\Pic_G(\algvar)$; 
    however $\vb$ also pulls back to $\vb_1$, 
    so to prove \eqref{eqn:maponpicardimages} is injective 
    it suffices to show that 
    $\Pic_G(\algvar/\Gamma) \to \Pic_G(\algvar)$ is injective. 

    There is a commutative diagram with exact rows 
    \cite[2.2]{kkv-1989}: 
\begin{equation}\label{eqn:diagram1} 
\begin{tikzcd}
    H^1(G, \mathcal O(\algvar/\Gamma)^\times) \ar[r]\ar[d]
      & \Pic_G(\algvar/\Gamma) \ar[r]\ar[d]
      & \Pic(\algvar/\Gamma)\ar[d] \\
    H^1(G, \mathcal O(\algvar)^\times) \ar[r]
      & \Pic_G(\algvar) \ar[r]
      & \Pic(\algvar) . 
\end{tikzcd}
\end{equation} 
The action of $G$ on $\mathcal O(\algvar)^\times$ 
    and $\mathcal O(\algvar/\Gamma)^\times$ 
    is trivial by assumption. 
Thus the left column is injective, 
    so it now suffices to show that 
    $\Pic(\algvar/\Gamma) \to \Pic(\algvar)$ is injective. 
Now the Hochschild--Serre spectral sequence 
\begin{equation} 
    H^p(\Gamma,H^q_{\mathrm{et}}(\algvar,\GG_m)) 
    \Rightarrow H^{p+q}_{\mathrm{et}}(\algvar/\Gamma,\GG_m) 
\end{equation} 
yields the exact sequence 
\begin{equation} 
    1 \longrightarrow H^1(\Gamma, \mathcal O(\algvar)^\times) 
    \longrightarrow \Pic(\algvar/\Gamma) 
    \longrightarrow \Pic(\algvar)^\Gamma 
    \longrightarrow H^2(\Gamma,\mathcal O(\algvar)^\times) 
    \longrightarrow \cdots . 
\end{equation} 
From this it suffices to show that 
$H^1(\Gamma, \mathcal O(\algvar)^\times) = 1$. 
As $\algvar$ is projective, 
$\mathcal O(\algvar)/\mathcal O(\algvar/\Gamma)$ 
is a field extension with Galois group $\Gamma$, 
so $H^1(\Gamma, \mathcal O(\algvar)^\times)= 1$ 
by Hilbert's theorem 90. 
\end{proof} 



Using the lemma we can give a simple criterion 
for descent for line bundles 
even when the group action is not free. 

\begin{proposition}\label{prop:imageofpi} 
The image of $\Pic(\algvar/\gp) \to \Pic_\gp(\algvar)$ 
is the subset of isomorphism classes of $G$-line bundles $\vb$ over $\algvar$ 
    satisfying the following condition: 
    \begin{enumerate}
        \item[$(\ast)$] 
            the stabilizer subgroup $\gp_P$ acts trivially 
            on $\vb_P$ 
            for every $P \in \algvar(\overline{\nf})$. 
    \end{enumerate}
\end{proposition} 

\begin{remark}
A result of Mumford \cite[Cor.~1.6]{MR1304906} says that 
    if $\algvar$ is normal and proper, 
    with an action of 
    a connected linear group $G$, 
    and $\vb$ is a $G$-line bundle on $\algvar$, 
    then some positive power 
    $\vb^{\otimes e}$ is $G$-linearizable. 
Proposition~\ref{prop:imageofpi} 
    is the analogous result for finite $G$, 
    with $e$ the exponent of $G$. 
\end{remark}

\begin{proof} 
In the algebraically closed setting 
    this is \cite[Prop.~4.2]{kkv-1989}. 
We will reduce to this case using Lemma~\ref{lemma:picardimages}. 
Pulling back $\lb$ to $\algvar_{\overline{\nf}}$ shows 
    the condition ($\ast$) is clearly satisfied if $\lb$ descends to $\algvar/G$. 

We first observe that $\algvar_{L}/G = (\algvar/G)_{L}$ 
    for any $\nf$-algebra $L$. 
Indeed $\algvar_{L} \to (\algvar/G)_{L}$ is $G$-invariant 
so we get a map $\algvar_L/G \to (\algvar/G)_{L}$. 
To show this affine map 
    is an isomorphism we must show that 
    $\mcO_{(\algvar/G)_L} \to \mcO_{\algvar_L/G}$ 
    is an isomorphism of $\mcO_{(\algvar/G)_L}$-algebras. 
In fact these sheaves are already equal on the level of presheaves, 
namely the presheaves $U \mapsto H^0(G,\mcO_\algvar(\pi^{-1}(U))) \otimes L$ 
    and $U \mapsto H^0(G,\mcO_\algvar(\pi^{-1}(U)) \otimes L)$ 
    sheafify to $\mcO_{(\algvar/G)_L}$ and $\mcO_{\algvar_L/G}$, respectively,  
and these are isomorphic by the universal coefficient theorem: 
    for any $G$-module $N$ and $n \geq 0$ 
    we have the exact sequence \cite[Prop.~4.18]{jantzen}: 
\begin{equation} 
    0 \longrightarrow H^n(G,N) \otimes L 
    \longrightarrow H^n(G,N \otimes {L})
    \longrightarrow \operatorname{Tor}^\nf_1(H^{n+1}(G,N),L) 
    \overset{=}{\longrightarrow} 0. 
\end{equation} 

Now if ($\ast$) holds then $\lb \otimes \overline{\nf}$ 
    is the pullback of a line bundle on 
    $\algvar_{\overline{\nf}}/G$. 
This bundle on 
    $\algvar_{\overline{\nf}}/G = (\algvar/G)_{\overline{\nf}}$ 
    descends to $(\algvar/G)_{L} = \algvar_L/G$ 
    for some finite Galois extension $L/\nf$. 
Applying Lemma~\ref{lemma:picardimages} to 
    $\algvar_L$ (regarded over $\nf$ via 
    the structure map $\algvar_L \to \Spec K \to \Spec \nf$) 
    and $\Gamma = \Gal(L/\nf)$ 
    shows $\lb$ is in 
    the image of $\Pic(\algvar/G)\to\Pic_G(\algvar)$. 
\end{proof} 

When the action of $G$ on $\algvar$ is free, 
the proof of Lemma~\ref{lemma:picardimages} 
shows that $\Pic(\algvar/G) \to \Pic(\algvar)$ is injective. 
Unfortunately the spectral sequence there 
    is unavailable for non-free actions, 
but there is nonetheless 
an easily checked criterion for injectivity. 

\begin{lemma}\label{lemma:uniquedescent} 
Assume that $Y$ is geometrically integral. 
Suppose that for each nontrivial character 
    $c \colon \gp \to \nf^\times$ 
    there is a point $P \in \algvar(\algclosure{\nf})$ with 
    $c(\gp_{P}) \neq 1$. 
    Then the homomorphism
\begin{equation} 
\Pic(\algvar/G) \to \Pic(\algvar) 
\end{equation} 
is injective. 
\end{lemma} 

The hypothesis holds for instance if $\algvar \to \algvar/G$ has 
a totally ramified $\overline{\nf}$-point. 

\begin{proof} 
The pullback homomorphism 
    $\Pic(\algvar/\gp) \to \Pic_\gp(\algvar)$ is injective 
    (since the $G$-linearization suffices to undo the pullback), 
while the kernel of the homomorphism $\Pic_\gp(\algvar) \to \Pic(\algvar)$ 
    which forgets the $\gp$-linearization 
    can be identified with the group of 
    $\gp$-linearizations of the trivial bundle, namely  
    $H^1(\gp,\GS{\algvar}^\times)$ \cite[(2.2)]{kkv-1989}. 
As $\algvar$ is proper and geometrically integral, 
this group can be identified with $\Hom(\gp,\nf^\times).$ 
However the pullback of a line bundle on $\algvar/G$ 
to $\algvar$ carries a unique $G$-linearization for which it descends. 
Indeed, if $\DLB$ is a $\gp$-line bundle on $\algvar$ 
which descends to $\algvar/\gp$, 
and $c$ is a nontrivial $\nf$-valued character,  
then 
    the $\gp_{P}$-action on 
    $(\DLB \otimes c)_{P}$ is nontrivial, 
    so $\DLB \otimes c$ does not descend to $\algvar/\gp$ 
    (Proposition~\ref{prop:imageofpi}). 
\end{proof} 

\subsection{Quotient of the regular representation}\label{sec:quotientreg}

Let $V$ denote the {regular representation of $\gp$}, 
and write $\PPP$ for the projective space 
$\PPP(V)=\Proj \Sym V^\vee$ of $V$ 
regarded as a variety over $\QQ$. 
Let $\pi$ denote the quotient map $\PPP \to \PPP/\gp$. 


\begin{proposition}\label{prop:pullbackcanonicaldivisors}
    $\pi^\ast \CB{\PPP/\gp} = \CB{\PPP}$ if $|G| \neq 2$ 
    and $\pi^\ast \CB{\PPP/\gp} = 2\CB{\PPP}$ if $|G| = 2$. 
\end{proposition}

\begin{proof} 
First observe $\PPP/\gp$ is normal since it is 
    the quotient of a normal variety by a finite group. 
Now suppose $\PPP \to \PPP/\gp$ is not \'etale in codimension one, 
    i.e.~there is some irreducible codimension one subvariety $D$ 
    in the support of $\Omega^1_{\PPP\to\PPP/\gp}$. 
Then some $g \neq 1$ fixes the generic point $\eta$ of $D$. 
In fact $D$ is a hyperplane since 
    otherwise $g$ would fix 
    $|\gp|-1$ linearly independent vectors in 
    the affine cone of $D$, 
    any line passing through $0 \in V$ would intersect 
    the affine cone of $D$ at a nonzero point, 
    and $g$ would be $1$. 
We see that $g$ fixes a codimension one hyperplane 
    (the affine cone of $D$), 
so $G$ contains a pseudoreflection. 

Suppose $|\gp| \neq 2$. 
Then the regular representation contains no pseudoreflections 
    and so $\PPP \to \PPP/\gp$ is \'etale in codimension one. 
We may take an open subset $U \subset \PPP$ 
    whose complement has codimension at least two 
    such that $\pi|_U$ is \'etale 
    and $\pi(U)$ is contained in 
    the smooth locus $V$ of $\PPP/\gp$. 
Then $\pi^\ast\Omega_{V/\QQ}^1$ and $\Omega_{\PPP/\QQ}^1$ 
    are isomorphic over $U$ 
    which means 
    $\pi^\ast \Omega_{(\PPP/\gp)/\QQ}^{|\gp|-1}$ and 
    $\Omega_{\PPP/\QQ}^{|\gp|-1}$ are isomorphic over $U$. 
This shows their supports 
    $\pi^\ast \CB{\PPP/\gp}$ and $\CB{\PPP}$ 
    are equal away from a codimension two closed subset 
    and are therefore equal everywhere. 

For $|G| = 2$ the map $\PPP \to \Proj \QQ[u,v]$ 
    given by $(u,v) = ((x+y)^2, (x-y)^2)$ 
    induces an isomorphism $\PPP/\gp \cong \Proj \QQ[u,v]$, 
and the anticanonical bundle on $\Proj \QQ[u,v]$ 
    pulls back to $-2K_\PP$. 
\end{proof} 

\begin{theorem}\label{thm:discriminantlinebundle} 
There is a line bundle $\DLB$ on $\PPP/\gp$, 
unique up to isomorphism, 
    satisfying $\pi^\ast \DLB = \ACB{\PPP}.$ 
If $|\gp| \neq 2$ then $\DLB+\CB{\PPP/\gp}$ 
    is torsion in 
    the divisor class group of $\PPP/\gp$ 
    and if $|\gp| = 2$ then 
    $2\DLB = -\CB{\PPP/\gp}$. 
\end{theorem} 

\begin{proof} 
Let $e$ be the exponent of $G$. 
Then $\Pic(\PPP)^{\times e}$ is contained in 
    the image of $\Pic(\PPP/\gp) \to \Pic(\PPP)$ 
    by Proposition~\ref{prop:imageofpi}, 
    and we conclude the existence of $\DLB$. 
Uniqueness follows from the existence of the totally ramified point 
    $[1:\cdots:1]$ for $\pi$ and Lemma~\ref{lemma:uniquedescent}. 
For the second assertion note that $\CB{\PPP/\gp}$ 
    is $\QQ$-Cartier 
    \cite[Prop.~5.20]{kollar-mori-1998} 
    so if $n\CB{\PPP/\gp}$ is Cartier 
    and $|\gp| \neq 2$ then 
    $$
    \pi^\ast(n(\DLB+\CB{\PPP/\gp}))  = 0. 
    $$ 
    However $\pi$ is totally ramified at $[1:\cdots:1]$ 
    so $\pi^\ast$ is injective (Lemma~\ref{lemma:uniquedescent}). 
\end{proof} 

\begin{remark}
With a bit more work one can show that 
$\DLB + \CB{\PPP/\gp}$ is $2$-torsion 
and that $\DLB + \CB{\PPP/\gp}$ is trivial 
if the Sylow $2$-subgroup of $\gp$ is trivial or non-cyclic. 
This shows that $\PPP/\gp$ is $\QQ$-Gorenstein of index $\leq 2$. 
\end{remark}


\begin{corollary}\label{cor:picardgroup}
    The Picard group of $\PPP/\gp$ is $\ZZ$. 
\end{corollary}

\begin{proof} 
    $\pi$ has a totally ramified point 
    so $\Pic(\PPP/\gp) \to \Pic(\PPP)$ 
    is injective 
    by Lemma~\ref{lemma:uniquedescent}, 
    however the pullback of $\DLB$ to $\PPP$ is 
    nontrivial. 
\end{proof} 

Let $n$ (resp. $e$) denote the order 
    (resp. exponent) of $\gp$. 

\begin{theorem}\label{thm:dlbintrinsicdefn}
    Let $\DLB_0$ be the ample generator of $\Pic(\PPP/\gp)$. 
    Then 
        $\DLB_0^{\frac{n}{e}} = \DLB$.  
\end{theorem}

\begin{proof} 
The pullback of $\DLB_0$ to $\PPP$ will be $\mathcal O(t)$ 
for the smallest positive integer $t$ such that 
    $\mathcal O(t)$ descends to $\Pic(\PPP/\gp)$ for some $G$-linearization. 
First equip $\mcO(t)$ with the $\gp$-linearization 
    coming from the natural $\gp$-action on $\mcO(1)$. 
If $c$ is any nontrivial character of $G$, 
    then $G$ will act on 
    the fiber of $\mathcal O(t) \otimes c$ 
    over $[1:\cdots:1]$ by $c$ 
    so $\mathcal O(t) \otimes c$ cannot descend (Proposition~\ref{prop:imageofpi}). 
This means that if $\mcO(t)$ descends for some $G$-linearization 
    it must be its natural $G$-linearization. 
Now if $g$ is an element of $G$ with order $d$ then 
    $\langle g \rangle$ acts as $\zeta^{-1}$ on the line spanned by 
        $\sum_{j = 0}^{d-1}
            \zeta^j \grb{\el^j}$ 
($\zeta$ a primitive $d$th root of unity), 
    so $g$ fixes the corresponding point of $\PPP(\overline{K})$ 
    and acts by $\zeta^t$ 
    on the fiber of $\mcO(t)$ over this point. 
The only way that $\zeta^t = 1$ 
    for all $\el \in \gp$ is if $e$ divides $t$, 
    and this condition is also sufficient for descent. 
We conclude that $\mcO(e)$ 
    descends to a line bundle $\DLB_0$ on $\PPP/\gp$ 
    and $\DLB_0^{\frac ne} = \DLB$. 
\end{proof} 

\subsection{\texorpdfstring{$\DLB$}{L} determines an immersion of \texorpdfstring{$\Xg$}{X}}


\begin{theorem}\label{thm:immersion}
$\DLB$ is globally generated 
and its global sections restrict to an immersion of $\Xg$ 
    into $\PP^{N}$, where $N+1$ is the dimension of 
    the linear subspace of homogeneous degree $|\gp|$ 
    $\gp$-invariants 
    in $\Sym V^\vee$. 
\end{theorem}

\begin{remark}
Using Molien's theorem one can show that 
    $N$ is equal to the sum over divisors $d$ of $|G|$ of 
    $\binom{2d-1}{d} \cdot \#\{g \in \gp: g\text{ has order }|G|/d\}$. 
\end{remark}


\begin{proof} 
    To show that the rational map 
    $\varphi \colon \PPP/\gp \dashrightarrow \PP^N$
    determined by $\DLB$ is defined on all of $\PPP/\gp$, 
    we may assume without loss of generality that 
    $\nf$ is algebraically closed. 
    Furthermore, as the property of being 
    a quasi-compact immersion 
    is stable under faithfully flat descent 
    we may also assume that 
    $\nf$ is algebraically closed for the claim that 
    $\varphi$ restricts to an immersion on $\Xg$. 

    Let $n$ denote the order of $\gp$. 
    We will show that 
    for any two distinct closed points 
    of $\PPP/\gp$ 
    there exists a global section of $\DLB$ 
    vanishing at one point but not the other. 
    Since pullback along $\pi$ induces an isomorphism 
        $H^0(\PPP/\gp,\DLB) \xrightarrow{\sim} 
        H^0(\PPP,\mcO(n))^{\gp}$,  
    it is the same to 
    find a $\gp$-invariant global section of $\mcO(n)$ 
    separating two arbitrary closed $\gp$-orbits in $\PPP$. 
    Suppose $\pt$, $\pttwo\in\PPP(\algclosure{\nf})$ are 
    not in the same $\gp$-orbit. 
    Choose a hyperplane $H\subset\PPP$ passing through $\pt$ 
    but disjoint from the orbit of $\pttwo$, and 
    let $t$ be a global section of 
    $\cO(1)$ whose zero locus is $H$. 
    Then
    \begin{equation}\label{eqn:globalsectionpq}
        \prod_{\el\in \gp}{^\el t}
    \end{equation}
    is a $\gp$-invariant global section of 
    $\cO(n)$ vanishing at $\pt$ but not at $\pttwo$. 
    The existence of such sections also  
    shows that $\DLB$ is globally generated 
    so $\varphi$ is defined on all of $\PPP/\gp$. 
    The same fact shows that $\varphi$ is injective 
    on closed points of $\PPP/\gp$.  


    We next show that $\varphi$ separates tangent vectors 
    away from isotropy. 
    It is equivalent to show $\tp$ separates tangent vectors 
    away from isotropy since the quotient map  
    $\PPP\to \PPP/\gp$ is an isomorphism on tangent spaces 
    away from isotropy. 
    Suppose $\pt\in\PPP(\algclosure{\nf})$ has 
    trivial isotropy group and 
    $0\neq v\in T_\pt(\PPP)$. 
    Choose a hyperplane $H\subset\PPP$ passing through $\pt$, 
    disjoint from $\{\el(\pt):\el\in \gp-\{1\}\}$, and 
    not tangent to $v$. 
    Let $t$ be a global section of $\cO(1)$ whose zero locus is $H$. 
    Then
    \[
    \prod_{\el\in \gp}{^\el t}
    \]
    is a $\gp$-invariant global section of 
    $\cO(n)$ vanishing at $\pt$ 
    such that $v$ is not tangent to its zero locus. 
    This implies $\tp_* v$ is not zero 
    (cf.~Remark II.7.8.2 of \cite{hartshorne}). 
    We conclude $\tp_*$ is injective on $T_\pt(\PPP)$. 
    Since $\gp$ acts freely on $\UG{\gp}$ 
    we conclude that $\varphi$ separates tangent vectors 
    on $\Xg$. 
    
    To conclude the proof 
    it suffices to show that the restriction 
    $\varphi|_{\Xg} \colon \Xg \to \varphi(\Xg)$ 
    of $\varphi$ to $\Xg$ is proper 
    and that $\varphi(\Xg)$ is 
    a locally closed subset of $\PP^N$. 
    This will imply that 
    $\varphi|_{\Xg}$ is a closed immersion 
    of $\Xg$ into an open subset of $\PP^N$ 
    by \cite[Prop.~12.94]{gortz-wedhorn}, 
    obtaining the desired conclusion that the composition 
    $\Xg \to \varphi(\Xg) \to \PP^N$ is an immersion. 
    Observe that 
    \begin{equation*}
    \begin{tikzcd}
        \Xg \arrow[r,"\subset"] 
            \arrow[d,"\varphi|_{\Xg}"] & 
            \PPP/\gp \arrow[d,"\varphi"] \\
        \varphi(\Xg) \arrow[r,"\subset"] & \PP^N 
    \end{tikzcd}
    \end{equation*}
    is a pullback diagram because $\varphi$ is injective.  
    As the property of being proper is 
    preserved under basechange 
    this shows that $\varphi|_{\Xg}$ is proper. 
    To see that $\varphi(\Xg)$ is locally closed 
    we first observe that 
    the subset of $\PPP$ where the group determinant 
    $\GDET{\gp}$ is invertible is equal to $\UG{\gp}$. 
    As $\UG{\gp}$ is a $\gp$-stable affine open subset, 
    its quotient $\Xg$ is an open subset of $\PPP/\gp$. 
    Then we have that 
    $\varphi(\Xg \sqcup \Xg^c)
    =
    \varphi(\Xg) \sqcup \varphi(\Xg^c)$
    which shows that 
    \begin{equation}\label{eqn:locallyclosed}
    \varphi(\Xg) = 
    \varphi(\PPP/\gp) \cap (\PP^N - \varphi(\Xg^c)).
    \end{equation}
    As $\varphi$ is a proper morphism it is closed 
    so \eqref{eqn:locallyclosed} 
    shows $\varphi(\Xg)$ is locally closed. 
\end{proof}



\section{Height formulas}\label{sec:heightformulas} 

In this section we prove 
Theorem~\ref{thm:thm1} and Theorem~\ref{thm:stablemultiplierorder}. 
For each rational prime $p$ 
let $|\cdot|_p$ denote the norm on $\algclosure{\QQ_p}$ 
satisfying $|p|_p=1/p$ and let $|\cdot|_\infty$ 
denote the usual complex absolute value. 
If $\nfe/\QQ$ is any finite extension, 
then the set of places $w$ of $\nfe$ 
over a fixed place $v$ of $\QQ$ 
is naturally in bijection with field homomorphisms 
$j \colon \nfe \to \algclosure{\QQ_v}$ up to isometry,  
and we write $|\cdot|_w$ for $|j(\cdot)|_v$. 
We write $\norm{\cdot}_\infty$ for 
    the canonical norm on Minkowski space $\nfe_\RR$. 
For any $y \in \nfe \subset \nfe_\RR$ it is given by 
\begin{equation}\label{eqn:minkowskinorm}
    \norm{y}_\infty=
        \left(\sum_{j: \nfe \to \CC}
            d_j
            |j(y) |_\infty^2
        \right)^{1/2}
\end{equation}
where $j$ runs over the complex embeddings of $\nfe$ 
up to isometry (one $j$ for each pair of complex embeddings) 
and $d_j$ is $1$ if $j$ is real and $2$ if $j$ is complex 
(cf.~\cite{neukirch}). 
For a product $L= \nfe\times \cdots \times \nfe$ 
    we extend the norm to $L_\RR$ 
    by setting 
\begin{equation}
    \norm{y}_\infty = 
    \left(\sum_{i=1}^{\dim_\nfe L}
    \norm{y_i}_\infty^2\right)^{1/2} . 
\end{equation}

\begin{theorem*}[Theorem~\ref{thm:thm1}]
Let $\phi \colon \PPP/G \to \PP^N$ be a non-constant morphism. 
Let $d_\phi$ be the degree of the composite map 
    $\PPP \to \PPP/\gp \xrightarrow{\phi} \PP^N$. 
Then for all $\pt = (L,x) \in \X(\QQ)$, 
\begin{equation}
    h(\phi(\pt)) =
    d_\phi\log{\norm{x}_\infty} - 
    \frac{d_\phi}{[\nfe:\QQ]}\log{N(J)} + O(1) 
\end{equation}
    for a bounded function $O(1)$. 
\end{theorem*}

\begin{proof} 
Recall that $\PPP/\gp$ has Picard rank one 
    (Corollary~\ref{cor:picardgroup}). 
As we are only proving \eqref{eqn:newmainformula} 
    up to $O(1)$, 
    the claimed formula follows for any morphism $\phi$ 
    once we have proven it for a single morphism $\phi$. 
We use $\lb$ to determine such a $\phi$ 
    as $\lb$ is globally generated 
    (Theorem~\ref{thm:immersion}). 
Then $d_\phi = |G|$ 
and the formula to be shown is that 
    $$h(\phi(P)) =
    |G|\log{\norm{x}_\infty} - 
    \frac{|G|}{[\nfe:\QQ]}\log{N(J)} + O(1).$$ 

Let $\varphi \colon L \to \CC$ 
    be any $\QQ$-algebra homomorphism. 
    By Proposition~\ref{prop:fiber}, 
    the fiber of $\pi$ over $(L,x)$ 
    consists of the $\gp$-translates of the unit given by 
    \begin{equation}
        \label{eqn:affinecoordsofunit}
        u = \sum_{\el \in \gp}
            \varphi(\el(x))\grb{\el^{-1}} \in \UG{\gp}(\nfe). 
    \end{equation}
We will use the Weil height $H$ on $\PPP(\algclosure{\QQ})$ 
    associated to $\mcO(1)$ 
    for which $H(u)$ is 
    \begin{equation}
    \prod_{w \in \INFPRIMES{\nfe}}
        \ds\left(\sum_{g \in \gp} 
            \left|
            \varphi(g(x))
            \right|_w^2
        \right)^{{d_w}/(2[\nfe:\QQ])}
    \prod_{w \in \FINPRIMES{\nfe}}
        \max_{g \in \gp}
                \left|
                \varphi(g(x))
                \right|_w^{{d_w}/[\nfe:\QQ]}
                =H_\infty(u)  H_f(u)
    \end{equation}
    where $d_w$ is the local degree of $\nfe$ at $w$. 
(We call this the \defn{standard metric on $\mathcal O(1)$}.) 
By the theory of heights, 
\begin{equation} 
    h(\phi(L,x)) = |G|\log H(u) + O(1)  
\end{equation} 
    for a bounded function $O(1)$. 

Let $d_\infty$ be 
    the local degree of $\nfe$ at any of its infinite places. 
The infinite component $H_\infty$ of the height is none other  
    than the canonical norm on Minkowski space $\nfe_\RR$ 
    (when evaluated on a unit of the form $u$). 
Indeed, $H_\infty(u)$ is equal to 
\begin{align*} 
    \prod_{w \in \INFPRIMES{\nfe}}
        \left(\sum_{g \in G}
            \left|
            \varphi(g(x))
            \right|_w^2
        \right)^{\frac{d_\infty}{2[\nfe:\QQ]}}
    &=
    \prod_{j: \nfe \to \CC}
        \left(\sum_{g \in \gp}
            \left|
            j(\varphi(g(x))) 
            \right|_\infty^2
        \right)^{\frac{d_\infty}{2[\nfe:\QQ]}} \\
    &=\prod_{j: \nfe \to \CC}
        \left(
        \sum_{i=1}^{\dim_\nfe L}
        \sum_{g \in \Gal(\nfe/\QQ)}
            \left|
            j(g(x_i)) 
            \right|_\infty^2
        \right)^{\frac{d_\infty}{2[\nfe:\QQ]}} \\
    &=
    \prod_{j: \nfe \to \CC}
        \left(
        \sum_{i=1}^{\dim_\nfe L}
        \sum_{j': \nfe \to \CC}
            d_\infty
            \left|
            j'(x_i) 
            \right|_\infty^2
        \right)^{\frac{d_\infty}{2[\nfe:\QQ]}} \\
    &=
        \left(
        \sum_{i=1}^{\dim_\nfe L}
        \sum_{j: \nfe \to \CC}
            d_\infty
            \left|
            j(x_i) 
            \right|_\infty^2
        \right)^{\frac 12} \\
    &= \norm{x}_\infty.
\end{align*} 

For the finite components observe that $H_f(u)$ is equal to 
    \begin{equation} 
        \prod_{w \in \FINPRIMES{\nfe}}
        \max_{g \in \gp}
                \left|
                \varphi(g(x))
                \right|_w^{{d_w}/[\nfe:\QQ]}
        =
        \prod_{w \in \FINPRIMES{\nfe}}
        \max_{y \in J}
            |y|_w^{{d_w}/[\nfe:\QQ]}
        =N(J)^{-1/[\nfe:\QQ]} . 
        \tag*{$\qed$} 
    \end{equation} 
\renewcommand{\qedsymbol}{}
\vspace{-\baselineskip}
\end{proof} 

Let $\ord{}_{\infty} = \ds\lim_{D \to \infty} T_D$ 
    denote the stable multiplier order of $\nfe$ 
    attached to a rational point $(L,x)$ of $\X$. 

\begin{theorem*}[Theorem~\ref{thm:stablemultiplierorder}] 
    \,\,\,
    \begin{enumerate}
        \item $\Spec T_\infty \sqcup \cdots \sqcup \Spec T_\infty$ 
            for $\dim_\nfe L$ many copies 
            is isomorphic to the fiber of $\PP \to \PP/G$ 
            over the $\ZZ$-point $(L,x) \to \PP/G$. 
        \item $T_\infty = T_{D}$ if $D \geq |G|-1$. 
    \end{enumerate}
\end{theorem*} 


\begin{proof} 
Let $\Spec A$ denote the fiber over $\PPP \to \PPP/\gp$ 
    over $(L,x) \in \Xg(\QQ)$, 
    where we regard $(L,x)$ as a $\ZZ$-point of $\PPP/\gp$ 
    by the valuative criterion of properness. 
Let $q \colon \Spec A \to \PPP$ denote the natural projection map. 
We have identifications 
\begin{equation}
    H^0(\PPP,\mcO(D)) = 
    \bigoplus_{m \in Mon_D} \ZZ m
\end{equation}
and 
\begin{equation}
    H^0(\Spec A, q^\ast \mcO(D) ) 
    = \sum_{m \in Mon_D} A m(\{g(x)\}_g). 
\end{equation}
We claim the natural map 
\begin{equation}
    \phi \colon 
    H^0(\PPP,\mcO(D))  \to  
    H^0(\Spec A, q^\ast \mcO(D) ) 
\end{equation}
is surjective if $D \geq |G|-1$. 
Indeed its cokernel $C$ is finite, 
since its image and codomain are both lattices 
    (full rank abelian groups) of $L$. 
However $C$ also has trivial $p$-torsion for each prime $p$ 
    by \cite[Lemma~2.1]{MR2144974}. 
Thus $C = 0$. 
We see that 
\begin{equation}\label{eqn:imageofphi}
    \im \phi 
    = \sum_{m \in Mon_D} \ZZ m(\{g(x)\}_g) 
    = \sum_{m \in Mon_D} A m(\{g(x)\}_g). 
\end{equation}
This equality shows that $\im \phi$ 
    is closed under multiplication by $A$. 
Since $q^\ast \mcO(D)$ is a line bundle over $\Spec A$, 
    the $A$-module $\im \phi$ is projective. 

Now apply $\varphi$ to both sides of \eqref{eqn:imageofphi} 
where $\varphi \colon L \to \nfe$ is any $\QQ$-algebra homomorphism. 
The homomorphism $\im \phi \to \varphi(\im \phi) = \lat{}^D$ splits, 
    which implies that $\lat{}^D$ is a direct summand 
    of $\im \phi$ and is therefore projective 
    as an $A$-module. 
The action of $A$ on $\lat{}^D$ factors 
    through $\varphi(A)$, 
    which shows $\lat{}^D$ is a projective $\varphi(A)$-module. 
This implies that $\varphi(A) = (\lat{}^D:\lat{}^D)$ 
    by localization. 
As $\ord{}_\infty = \lim_{D \to \infty} (\lat{}^D:\lat{}^D)$ 
    we have $\varphi(A) = \ord{}_\infty$. 
Since $A \cong \varphi(A)^{\dim_\nfe L}$ the result follows. 
\end{proof} 

\begin{corollary} 
    For any integer $D \geq |G|-1$, 
    $N(I^D) = N(J^D)$. 
\end{corollary} 

\begin{proof} 
$I^D$ is the module of global sections 
    of the line bundle on $\Spec T_\infty$ 
    obtained by pullback of $\mcO(D)$ 
    along $\Spec T_\infty \to \PPP$ 
    and is therefore an invertible $T_\infty$-module. 
By localization we see that 
\begin{equation} 
    N(I^D) = [T_\infty : I^D] = 
    [O_K : O_KI^D] = [O_K : J^D] = N(J^D).
    \tag*{$\qed$} 
    \end{equation} 
\renewcommand{\qedsymbol}{}
\vspace{-\baselineskip}
\end{proof} 

\begin{remark}
An elementary argument shows that 
    $N(J)N(I)^{-1}$ 
    is bounded from above by the index of $T_1$ in $O_\nfe$ 
    and from below by $1$. 
\end{remark}


\section{Descent for metrics}  
\label{sec:descentmetrics} 

In this final section 
    we take up the problem of descending 
    metrized line bundles through 
    finite quotient maps. 
As is well-known, 
    the theory of metrized line bundles refines 
    the theory of heights 
    and is important for comparing  
    height functions not just up to an $O(1)$. 
With an eye towards future applications, 
    we take a step in this direction. 
The main result in this section is Theorem~\ref{thm:pushforward}, 
    which implies a refinement of Theorem~\ref{thm:thm1}. 
Let $\algvar$ be a projective $\gp$-variety over $\QQ$. 
Our result says that if $\lbtwo$ is any metrized $\gp$-line bundle 
    over $\algvar$ 
    for which $\gp$ acts by isometries, 
    then $\lbtwo^{\otimes |\gp|}$ descends to 
    a metrized line bundle $\lb$ on $\algvar/\gp$ such that 
    the associated height functions satisfy 
\begin{equation}
    h_\lb(x)  = |\gp|h_{\lbtwo}(y) 
\end{equation}
where $y$ is any preimage of 
$x \in (\algvar/\gp)(\algclosure{\QQ})$. 

Defining this metric on $\DLB$ as a set-theoretic function is easy: 
    the fiber $x^\ast \DLB$ can be identified 
    with the fiber of $\lbtwo^{\otimes |\gp|}$ 
    over any preimage of $x$ 
    and inherits the metric we already have there. 
This will be independent of the choice of preimage 
if $G$ acts by isometries on the metric on $\lbtwo$. 
The difficulty is showing that this function 
satisfies standard technical requirements on metrics, 
specifically the so-called adelic compatibility condition 
(see (3) in the definition below).

\begin{definition}\label{defn:adelicmetric} 
    A \defn{$v$-adic metric} on $\lbtwo$ is a collection 
    $\norm{\,\cdot\,}_v=
    (\norm{\,\cdot\,}_{v,y})_{y\in\algvar(\algclosure{\QQ_v})}$ 
    where $\norm{\,\cdot\,}_{v,y}$ is a norm 
    on $y^\ast \lbtwo$ satisfying 
    $\norm{cu}_{v,y} = |c|_v\norm{u}_{v,y}$ 
    for all $c \in \algclosure{\QQ_v}$ and $u \in y^\ast\lbtwo$. 
    A collection 
    $(\norm{\,\cdot\,}_v)_{v \in \PRIMES{\QQ}}$ 
    of $v$-adic metrics 
    is called an \defn{adelic metric} on $\lbtwo$ 
    if the following conditions are satisfied 
    for all $v \in \PRIMES{\QQ}$ and 
    $y \in \algvar(\algclosure{\QQ_v})$: 
\begin{enumerate}
    \item 
        $(y^\ast\lbtwo,\norm{\,\cdot\,}_{v,y}) \xrightarrow{\sigma} 
        ((\sigma y)^\ast\lbtwo,\norm{\,\cdot\,}_{v,\sigma y})$ 
        is an isometry 
        for every $\sigma \in \gal(\algclosure{\QQ_v}/\QQ_v)$, 
    \item $U(\algclosure{\QQ_v}) \to \RR \colon 
        y \mapsto \norm{s(y)}_{v,y}$ is continuous 
        for any open subset $U$ of $\algvar$ and $s \in {\lbtwo}(U)$, 
    \item $\lbtwo$ admits a generating set of global sections 
    $\{s_1,\ldots,s_m\}$, independent of $y$ and $v$,  
    with the property that 
    for all but finitely many $v \in \PRIMES{\QQ}$, 
\begin{equation} 
    \label{eqn:localheightfromgloballygeneratingset}
    \norm{s(y)}_{v,y} = 
    \left(\max_{1 \leq j \leq m}
            \left|
            \frac{s_j(y)}{s(y)} 
            \right|_{v}
    \right)^{-1}
\end{equation} 
    for any local section $s$ that is nonzero at $y$. 
\end{enumerate}
\end{definition}

When there is no need for disambiguation 
we write $\norm{\cdot}_v$ for $\norm{\cdot}_{v,y}$. 

Now let $\nfe/\QQ$ be a number field. 
The set of places $w$ of $\nfe$ over a fixed place $v$ of $\QQ$ 
is naturally in bijection with the set of field homomorphisms 
$j \colon \nfe \to \algclosure{\QQ_v}$ up to isometry. 
Given a point $y\in \algvar(\nfe)$ 
let $j(y) \in \algvar(\algclosure{\QQ_v})$ denote the point 
$\Spec \algclosure{\QQ_v}\to\Spec \nfe \xrightarrow{y} \algvar$. 
Let us write $\norm{\cdot}_{w,y}$ for the norm obtained on 
the fiber $y^\ast \lb$ (non-canonically isomorphic with $\nfe$) 
by pulling back $\norm{\cdot}_{v,j(y)}$ along the natural map 
$$y^\ast \lb \to 
y^\ast \lb \otimes_j \algclosure{\QQ_v}
= j(y)^\ast \lb.$$ 

The height of $y\in\algvar(\algclosure{\QQ})$ is defined by 
\begin{equation}
H_{\lb}(y) = 
\prod_{w \in \PRIMES{\nfi}}
    \norm{s(y)}_{w,y}^{-{d_w}/[\nfe:\QQ]}
\end{equation}
where $s\in\lb$ is any local section that is nonzero at $y$ 
and $\nfe/\QQ$ is any finite extension which $s$ and $y$ 
are defined over. 
(The $d_w$ in the exponent 
ensures this is independent of $s$ 
by the product formula 
and the $[\nfe:\QQ]$ ensures 
this is independent of $\nfe$.) 


The following lemma helps with finding global sections 
which satisfy \eqref{eqn:localheightfromgloballygeneratingset}. 

\begin{lemma} 
    \label{lemma:definingglobalsections}
    Let $\norm{\,\cdot\,}_v$ be a $v$-adic metric on $\lb$. 
    If $\{s_1,\ldots,s_m\}$ is a set of global sections of $\lb$ 
    satisfying 
    \begin{equation}
        \label{eqn:lemmamaxisone}
        \max_{1 \leq i \leq m} 
            \norm{s_i(x)}_{v} = 1 
    \end{equation}
    for all $x \in \algvar(\algclosure{\QQ_v})$, 
    then $\{s_1,\ldots,s_m\}$ globally generates $\lb$. 
    Furthermore, 
    for any local section $s$ that is nonzero at $x$, 
    \begin{equation}
        \norm{s(x)}_v = 
    \left(\max_{1 \leq i \leq m}
            \left|
            \frac{s_i(x)}{s(x)} 
            \right|_v
    \right)^{-1}.
    \end{equation}
\end{lemma} 

\begin{proof} 
It is clear that the $\{s_1,\ldots,s_m\}$ globally generate 
    $\lb$ from \eqref{eqn:lemmamaxisone}. 
    For any $i \in \{1,\ldots,m\}$ with $s_i(x) \neq 0$ 
    we have 
\begin{equation}
\norm{s(x)}_v = 
    \left|
    \frac{s(x)}{s_i(x)} 
    \right|_v
    \norm{s_i(x)}_v
    \leq
    \left|
    \frac{s(x)}{s_i(x)} 
    \right|_v.
\end{equation}
    By \eqref{eqn:lemmamaxisone}, 
    equality is obtained for some $i$ and therefore 
    \begin{equation}
        \norm{s(x)}_v = 
    \min_{i:s_i(x)\neq 0}
    \left|
    \frac{s(x)}{s_i(x)} 
    \right|_v
    =
    \left(\max_{1 \leq i \leq m}
    \left|
    \frac{s_i(x)}{s(x)} 
    \right|_v
    \right)^{-1}. 
    \tag*{$\qed$} 
\end{equation}
\renewcommand{\qedsymbol}{}
\vspace{-\baselineskip}
\end{proof} 

Let $\lbtwo$ be a metrized $\gp$-line bundle over $\algvar$. 

\begin{definition}
We say that \defn{$\gp$ acts by isometries} 
if for all $v \in \PRIMES{\QQ}$, 
$g \in \gp$, and $x \in \algvar(\algclosure{\QQ_v})$, 
the linear map on fibers 
\begin{align} 
    (x^\ast \lbtwo,\norm{\,\cdot\,}_{w,x}) 
    \xrightarrow{g} ((xg)^\ast\lbtwo,\norm{\,\cdot\,}_{w,xg})
\end{align} 
is an isometry. 
\end{definition} 

Let $\pi \colon \algvar \to \algvar/\gp$ 
    denote the quotient map. 
By Proposition~\ref{prop:imageofpi}, there is 
    a unique line bundle $\lb$ on $\algvar/\gp$ such that 
    $$\pi^\ast \lb \cong \lbtwo^{\otimes |\gp|}$$ 
    as $\gp$-line bundles. 
Since $\lbtwo^{\otimes |\gp|}$ carries a natural metric 
    induced from $\lbtwo$, we can ask whether 
    this metric descends to $\lb$. 

\begin{theorem}\label{thm:pushforward} 
Suppose $G$ acts by isometries on $\lbtwo$. 
Then there is a unique adelic metric on $\lb$ 
    which pulls back to the induced metric on $\lbtwo^{\otimes |\gp|}$. 
\end{theorem} 

Thus if $x \in (\algvar/\gp)(\algclosure{\QQ})$ and 
$y \in \algvar(\algclosure{\QQ})$ is any preimage, 
then 
    $H_{\lb}(x) 
    =
        H_{\lbtwo}(y)^{|\gp|}$. 

\begin{remark} 
The reason for descending the metric on 
    the $|G|$-fold tensor product rather than 
    the original metric 
    is that the global sections in 
    \eqref{eqn:localheightfromgloballygeneratingset} 
    have no reason to be $\gp$-invariant. 
Our remedy is to symmetrize the global sections 
by taking the product \eqref{eqn:sectionsymmetrized}
    over their $\gp$-orbits 
    which necessitates replacing $\lbtwo$ with $\lbtwo^{\otimes |\gp|}$. 
This suffices for our application 
    to $\lbtwo = \mcO(1)$ and the line bundle $\lb$ 
    constructed in \S\ref{sec:three}. 
\end{remark} 

\begin{proof} 
Let $x \in (\algvar/\gp)(\algclosure{\QQ_v})$. 
Let $s$ be any local nonvanishing section at $x$ and suppose
    $y \in \algvar(\algclosure{\QQ_v})$ maps to $x$. 
In order for the metric on $\DLB$ to pull back to 
    the metric on $\lbtwo^{\otimes |\gp|}$ 
    it is clear that we must define 
    the norm on $x^\ast \lb$ to be 
\begin{equation} 
    \norm{s(x)}_{v,x} \coloneqq \norm{(\pi^\ast s)(y)}_{v,y}, 
\end{equation} 
and so uniqueness is clear. 
This is independent of the choice of $y$ 
    since $\gp$ acts by isometries. 
It is straightforward to show that the collection 
    of $v$-adic metrics this defines 
    satisfies the first two conditions 
    (Galois invariance and continuity) 
    of an adelic metric. 

For the third condition 
    we will construct a generating set of global sections 
    $\{s_1,\ldots,s_m\}$ of $\lb$ such that 
\begin{equation} 
    \label{eqn:propclaim}
    \norm{(\pi^\ast s)(y)}_{v,y}
    = 
    \left(\max_{1 \leq i \leq m}
            \left|
            \frac{s_i(x)}{s(x)} 
            \right|_v
    \right)^{-1} 
\end{equation} 
    Let $\{t_1,\ldots,t_n\} \subset H^0(\algvar,\lbtwo)$ 
    be a generating set of global sections that defines 
    by \eqref{eqn:localheightfromgloballygeneratingset} 
    the $v$-adic components of the adelic metric on $\lbtwo$ 
    for all places $v \not \in S$, where 
    $S \subset \PRIMES{\QQ}$ is any finite set 
    containing the archimedean place. 
Fix a place $v\not \in S$ 
    and a point $y \in \algvar(\algclosure{\QQ_v})$. 
Then  
    \begin{equation}
        \label{eqn:normisone}
        \max_{1 \leq i \leq n} \norm{t_i(y)}_{v}
        =\max_{i: t_i(y) \neq 0}
        \min_{j: t_j(y) \neq 0}
        \left|
        \tfrac{t_j(y)}{t_i(y)}
        \right|_v
        =1.
    \end{equation}
Consider the set 
\begin{equation} 
    \label{eqn:defnA}
    A_y:= \{(i,g) \in [n] \times \gp : \norm{({^\el t_i})(y)}_{v} = 1 \}.
\end{equation} 
Since $\gp$ acts by isometries, 
\begin{equation} 
    \norm{({^g t_i})(y)}_{v} = 
    \norm{g(t_i(g^{-1} y))}_{v,y} =
    \norm{t_i(g^{-1} y)}_{v,g^{-1} y} 
\end{equation} 
    so from \eqref{eqn:normisone}, 
\begin{equation} 
    \label{eqn:maxisone}
    \max_{i,g} \norm{({^g t_i})(y)}_{v,y} = 
    \max_{i,g}  \norm{t_i(g^{-1} y)}_{v,g^{-1} y} = 1.
\end{equation} 
This shows $A_y$ is not empty so let $i$ be any fixed element of $[n]$ 
    with $(i,g)\in A_y$ for some $g \in \gp$. 
    Now define the sets 
\begin{equation} 
    B_y:= \{g : (i,g) \in A_y \}, \,\,\, H_y := \{h : B_yh = B_y \}. 
\end{equation} 
For each orbit $gH_y \in B_y/H_y$ choose a positive integer $a_{gH_y}$ 
    subject to the condition that 
\begin{equation} 
    \sum_{gH_y \in B_y/H_y} a_{gH_y} = [\gp:H_y] .
\end{equation} 
This is always possible since $H_y$ acts freely on $B_y$. 
It is easily verified that 
\begin{equation} 
    \prod_{h \in B_y}
        ({^{h}t_i})^{a_{h^{-1}H_y}} 
\end{equation} 
    is an $H_y$-invariant global section of $\lbtwo^{\otimes|\gp|}$. 
We define the global section 
\begin{equation} 
    \label{eqn:sectionsymmetrized}
    s'_y:=
    \sum_{gH_y \in \gp/H_y}
    \prod_{h \in B_y}
        ({^{gh}t_i})^{a_{h^{-1}H_y}}
\end{equation} 
which is manifestly $\gp$-invariant. 
Since 
    $H^0(\algvar,\lbtwo^{\otimes |\gp|})^\gp \cong H^0(\algvar/\gp,\lb)$ 
    (by pulling back along $\pi$) 
    the section $s'_y$ determines 
    a unique global section $s_y''$ of $\lb$ such that 
    $\pi^\ast s_y'' = s_y'$. 
    {\textit{A priori}} the section $s_y''$ depends on 
    $y$, $i$, $v$, and 
    the integers $(a_{gH_y})_{gH_y \in B_y/H_y}$, 
    but as $y$ varies over $\algvar(\algclosure{\QQ_v})$
    for all $v\not \in S$ 
    there are only finitely many possibilities for 
    $A_y$, $i$, and the $(a_{gH_y})_{gH_y \in B_y/H_y}$,  
    and these suffice to determine $s_y''$. 
    Let $\{s_1,\ldots,s_m\}$ denote 
    the set of global sections of $\lb$ arising in this way. 

Let $s_i' = \pi^\ast s_i$ for $1 \leq i \leq m$. 
First note that 
    $\frac{s_i(x)}{s(x)} 
    =
    \frac{s_i'(y)}{(\pi^\ast s)(y)}$ and thus 
\begin{equation} 
    \max_{1 \leq i \leq m}
            \left|
            \frac{s_i(x)}{s(x)} 
            \right|_v
            =
    \max_{1 \leq i \leq m}
            \left|
            \frac{s_i'(y)}{(\pi^\ast s)(y)} 
            \right|_v . 
\end{equation} 
Thus by Lemma~\ref{lemma:definingglobalsections}, 
if 
\begin{equation}
    \label{eqn:desiredmaximum}
    \max_{1 \leq i \leq m} 
        \norm{s_i'(y)}_v = 1 
\end{equation}
    for all $y \in \algvar(\algclosure{\QQ_v})$ 
    then \eqref{eqn:propclaim} holds. 
By construction, for each $s' \in \{s_1',\ldots,s_m' \}$ 
there is a global section $t \in \{t_1,\ldots,t_n\}$, 
a subset $B \subset \gp$, and a subgroup $H \leq \gp$ 
acting freely on $B$, and 
integers $a_{gH}$ for each coset 
$gH \in B/H$ such that 
\begin{equation}
    \label{eqn:expressionfors}
    s' =
    \sum_{gH \in \gp/H}
    \prod_{h \in B}
        ({^{gh}t})^{a_{h^{-1}H}}. 
\end{equation} 
Since the set $A_y$ defined by \eqref{eqn:defnA} is nonempty, 
there are $g \in B_y$ and $1 \leq i \leq n$ such that 
$\norm{({^g t_i})(y)}_{v} = 1$. 
By \eqref{eqn:maxisone}, $\norm{({^g t_i})(y)}_{v}<1$ for any $g \not \in B_y$,  
and therefore a single monomial in the sum 
\eqref{eqn:expressionfors} realizes the maximum. 
By the ultrametric inequality we have 
\begin{equation} 
    \max_{1 \leq i \leq m} \norm{s_i'(y)}_v
    =
    \prod_{h \in B_y}\big|\!\big|{({^{h}t_i})(y)}\big|\!\big|_v^{a_{h^{-1}H_y}} = 1 . 
    \tag*{$\qed$} 
\end{equation} 
\renewcommand{\qedsymbol}{}
\vspace{-\baselineskip}
\end{proof} 


Let $\lb$ denote the line bundle on $\PPP/\gp$ 
constructed in \S\ref{sec:three}. 

\begin{corollary}\label{cor:mainthmrefinement}
There is a unique adelic metric on $\DLB$ 
    which pulls back to the standard metric on 
    $\mathcal O(1)^{\otimes |\gp|}$. 
The height $h_{AC}$ induced by this metric satisfies 
    $$h_{AC}(L,x) = 
    |G|\log{\norm{x}_\infty} - 
    \frac{|G|}{[\nfe:\QQ]}\log{N(J)}$$ 
    on any rational point $(L,x) \in \X(\QQ)$. 
\end{corollary}

\begin{proof}
The proof of Theorem~\ref{thm:thm1} shows that 
    if $u \in \UG{\gp}$ is any preimage of $(L,x)$, then 
    $$h(u) = 
    \log{\norm{x}_\infty} - 
    \frac{1}{[\nfe:\QQ]}\log{N(J)}$$ 
    where $h$ is the height on $\PPP$ associated to 
    the standard metric on $\mathcal O(1)$. 
By the last theorem 
    there is a unique adelic metric on $\DLB$ 
    which pulls back to the metric on 
    $\mathcal O(1)^{\otimes |\gp|}$ 
    and the associated height $h_{AC}$ will satisfy 
    $h_{AC} (L,x) = |\gp|h(u)$. 
\end{proof}

In the next two examples we take $H_{AC} = e^{h_{AC}}$. 

\begin{example}[$\ord{x}$ Gorenstein, $L = \nfe$] 
If $\ord{x} = (\lat{x}:\lat{x})$ is Gorenstein, 
then $\lat{x}$ is $T$-projective 
(cf.~e.g.~\cite[4.2]{jensen-thorup}). 
(For instance, if $\ord{x}$ is monogenic then it is Gorenstein. 
    More generally, $T$ is Gorenstein if and only if 
    its different is invertible.) 
    Then $N(\lat{x}) = [\ord{x}:\lat{x}] 
    = \sqrt{\disc{\lat{x}}\disc{\ord{x}}^{-1}}$ 
    and 
    $$H_{AC}(L,x) = 
    \norm{x}_\infty^{|\gp|} 
    \sqrt{\disc{\ord{}}\disc{\lat{x}}^{-1}}.$$ 
If additionally $L$ is totally real and $x$ is self-dual, 
then 
\begin{equation}
    H_{AC}(L,x) = \sqrt{\disc{\ord{}}}. 
\end{equation}
\end{example} 

\begin{example}[$L = \QQ \times \cdots \times \QQ$] 
    $T = \ZZ$ and $I = \ell \ZZ$ 
    for some $\ell \in \QQ^{>0}$. Then 
    \begin{equation*} 
        H_{AC}(\QQ \times \cdots \times \QQ,x) =
        \left(\frac{\sqrt{x_1^2+\cdots+x_{|G|}^2}}{\ell}
        \right)^{|\gp|}.
    \end{equation*}
\end{example} 


\section*{Acknowledgements} 
The authors thank the following people for 
helpful discussions: 
Hyman Bass, Jordan Ellenberg, 
Wei Ho, Emanuel Reinecke, and Kannappan Sampath.

\bibliography{draft}
\bibliographystyle{abbrv}

\end{document}